\documentclass[english]{extarticle}
\usepackage[T1]{fontenc}
\usepackage[latin9]{inputenc}
\usepackage{geometry}
\usepackage{float}
\usepackage{amsmath}
\usepackage{amsthm}
\usepackage{amssymb}
\usepackage{graphicx}
\usepackage{mathtools}

\makeatletter
\numberwithin{equation}{section}
\numberwithin{figure}{section}
\theoremstyle{plain}
\newtheorem{thm}{\protect\theoremname}
\theoremstyle{plain}
\newtheorem{lem}[thm]{\protect\lemmaname}
\theoremstyle{remark}

\theoremstyle{plain}
\newtheorem{cor}[thm]{\protect\corollaryname}

\usepackage{lettrine} 
\let\myFoot\footnote
\renewcommand{\footnote}[1]{\myFoot{#1\vspace{0.8mm}}}
\usepackage{lmodern}
\usepackage[T1]{fontenc}
\usepackage{concmath}
\usepackage{ae,aecompl}
\usepackage[T1]{fontenc}
\usepackage{amsmath}
\usepackage{graphicx}

\usepackage{tikz}

\makeatother

\usepackage{babel}
\providecommand{\corollaryname}{Corollary}
\providecommand{\lemmaname}{Lemma}
\providecommand{\remarkname}{Remark}
\providecommand{\theoremname}{Theorem}
\begin{document}
\title{A variant of Cauchy's argument principle for analytic functions which applies to curves containing zeroes}
\author{Maher Boudabra, Greg Markowsky}
\maketitle
\begin{abstract}
    It is known that the Cauchy's argument principle, applied to an holomorphic function $f$, requires that $f$ has no zeros on the curve of integration. In this short note, we give a generalization of such a principle which covers the case when $f$ has zeros on the curve, as well as an application.
\end{abstract}

\section{Introduction, statement of results, and an application}

The argument principle, one of the fundamental results in complex analysis, can be formulated as follows (see \cite{Rudin2001} or \cite{remmert2012theory}).
\begin{thm} \label{classarg}
Suppose that $\gamma$ is a smooth Jordan curve in a domain $U$, and a function $f$ is analytic on $U$ with no zeroes on $\gamma$. Then the number of zeroes of $f$ inside $\gamma$ is equal to 

\begin{equation} \label{fprime}
\frac{1}{2 \pi i} \int_\gamma \frac{f'(z)}{f(z)}dz.
\end{equation}

This is also equal to the winding number of the curve $f(\gamma)$ around $0$.
\end{thm}

The equation (\ref{fprime}) shows the importance of the condition that $f$ not vanish on $\gamma$, since otherwise the integral would diverge, but it is the final statement that we will focus on. We have found the following variant on this statement which holds without the requirement that $f$ be nonzero on $\gamma$.

\begin{thm}
\label{thm:Let--be} Let $f$ be a non zero holomorphic function defined
on an open domain $U$ and $\gamma$ be a smooth Jordan curve lying
inside $U$. Then for any line $L$ passing through the origin there
exist at least $2m+\lambda$ distinct points on $\gamma$ mapped to $L$ by $f$, where
$m$ is the number of zeros of $f$ inside $\gamma$ and $\lambda$
is the number of zeros of $f$ on $\gamma$, all counted according to
multiplicities.
\end{thm}

We note that Theorem \ref{thm:Let--be} is an immediate consequence of the Theorem \ref{classarg} when there are no zeroes on $\gamma$, since any line will be intersected at least twice by $f(\gamma)$ each time that $f(\gamma)$ winds around $0$. Naturally the difficulty arises when there are zeroes on $\gamma$. A first attempt may be to factor out the zeroes, for instance writing $f(z)=g(z)h(z)$, where $g(z)$ is a polynomial with zeroes only on $\gamma$ and $h$ does not vanish on $\gamma$, and then to try to apply Theorem \ref{classarg} to $h$ and add the zeroes from $g$. However this approach fails, as the intersections of $f(\gamma)$ with $L$ are in general not respected by the factorization, and in any event zeroes of $g$ on $\gamma$ of high order must contribute many intersections of $f(\gamma)$ with $L$.

We will prove Theorem \ref{thm:Let--be} in the next section, but first we discuss the motivation for this result, in particular the application which inspired it. We thank Mohammed Zerrak for bringing this problem to our attention.

{\bf Problem:} Let $a_0, \ldots ,a_n$ be a sequence of real numbers, and for simplicity assume $a_0, a_n \neq 0$ (this requirement can easily be removed if required). Let $P(\theta) = \sum_{j=0}^n a_j \cos(j \theta)$ and $Q(\theta) = \sum_{j=0}^n a_{n-j} \cos(j \theta)$. Let $Z_P$ be the number of $\theta \in [0,2\pi)$ such that $P(\theta) = 0$, and $Z_Q$ defined analogously. Show that $Z_P + Z_Q \geq 2n$.

A quick solution using complex analysis can be provided, as follows. Let $f(z) = \sum_{j=0}^n a_j z^j$. Then $P(\theta) = \Re(f(e^{i\theta}))$, and it may be checked that $Q(\theta) = \Re(g(e^{i\theta}))$, where $g(z) = z^n f(\frac{1}{z})$. If there are no zeroes of $f$ on the unit circle $\{|z|=1\}$, then by Theorem \ref{classarg} the curve $\Re(f(e^{i\theta}))$ will intersect the imaginary axis at least $2m_f$ times, where $m_f$ is the number of zeroes (counting multiplicities) of $f$ inside the unit disk $\{|z|<1\}$, and the analogous statement holds for $g$. We see that $Z_P + Z_Q \geq 2(m_f + m_g)$. However, the number of zeroes of $g$ in $\{|z|<1\}$ is the same as the number of zeroes of $f$ in  $\{|z|>1\}$, and since $f$ has no zeroes on the unit circle we see that $m_f+m_g = n$. The result therefore follows in this case.

The role of Theorem \ref{thm:Let--be} is to extend this solution in the case that $f$ has zeroes on the unit circle. Let $\lambda$ denote the sum of the multiplicities of the zeroes of $f$ on $\{|z|=1\}$, and note that the conjugates of these zeroes must be zeroes of $g$ with the same multiplicities. Applying Theorem \ref{thm:Let--be}, we see that $Z_P+Z_Q = 2m_f + \lambda + 2m_g + \lambda$, and as $m_f + m_g + \lambda = n$ the result follows.

As a final comment before proving Theorem \ref{thm:Let--be}, we point out that the number $2m+\lambda$ is the best possible, as the following examples show. Let $\gamma$ be the unit circle, and let $f(z) = (z+1)^n$. Then, taking $z = e^{i \theta}$, we can check that $(e^{i \theta} + 1)^n = 2^n e^{in\theta/2} \cos^n(\theta/2)$, so the number of intersections that $f(\gamma)$ has with the real axis is the same as the number of zeroes of $\Im(e^{in\theta/2})=\sin(n\theta/2)$ in $[0,2\pi)$, and this is $n$. This shows essentially that the $\lambda$ in the expression $2m + \lambda$ is sharp. The $2m$ is even easier, as we may take $\gamma$ again as the unit circle and $f(z) = z^n$, and $f(e^{i\theta})$ will intersect any line exactly $2n$ times as $\theta$ ranges from $0$ to $2\pi$.

\iffalse
\begin{cor}
If $f$ is a polynomial of degree $n$ then $f(z)$ and $z^{n}f(\frac{1}{z})$
together cross any line $L$through the origin more than $2n$ times
when $z$ ranges over the unit circle.
\end{cor}

\begin{proof}
Set $h(z)=z^{n}f(\frac{1}{z})$ and consider the following numbers
\[
\begin{alignedat}{1}m_{f} & :=\#\{z\in\mathbb{D}\mid f(z)=0\}\\
\lambda & :=\#\{z\in\partial\mathbb{D}\mid f(z)=0\}=\#\{z\in\partial\mathbb{D}\mid g(z)=0\}\\
m_{g} & :=\#\{z\in\mathbb{D}\mid f(z)=0\}
\end{alignedat}
\]
In particular $m_{f}+\lambda+m_{g}=n$. Using (\ref{thm:Let--be}),
the number of crossings for both $f$ and $g$ with $L$ is then bounded below
by 
\[
2m_{f}+\lambda+2m_{g}+\lambda=2n.
\]
\end{proof}

\fi

\section{Proof of Theorem \ref{thm:Let--be}}
\begin{proof}
The number of zeros of $f$ on and inside $\gamma$ is finite as the
latter is a Jordan curve. Let $z_{1},...,z_{k}$ be the roots of $f$
on $\gamma$, where we denote by $\lambda_{j}$ the multiplicity of
each $z_{j}$ and $D_{j}=D_j(\varepsilon)$ be the disc of radius $\varepsilon$ 
centered at $z_{j}$. The radius $\varepsilon$ is chosen small enough
so $D_{j}$ remains inside $U$ and contains no zeroes of $f$ other than $z_j$. Now consider the Jordan curve $\gamma_{\varepsilon}$
constructed from $\gamma$ by replacing each part $\gamma\cap D_{j}$
by the arc of the circle $\partial D_{j}$ lying outside of $\gamma$; in order to guarantee that $\gamma_\varepsilon$ is itself a Jordan curve we may need to decrease $\varepsilon$ further. The figure illustrates $\gamma_\varepsilon$.

\begin{figure}[H]
\centering{}~~~~~~~~~~~~~~\includegraphics[width=7cm,height=7cm,keepaspectratio]{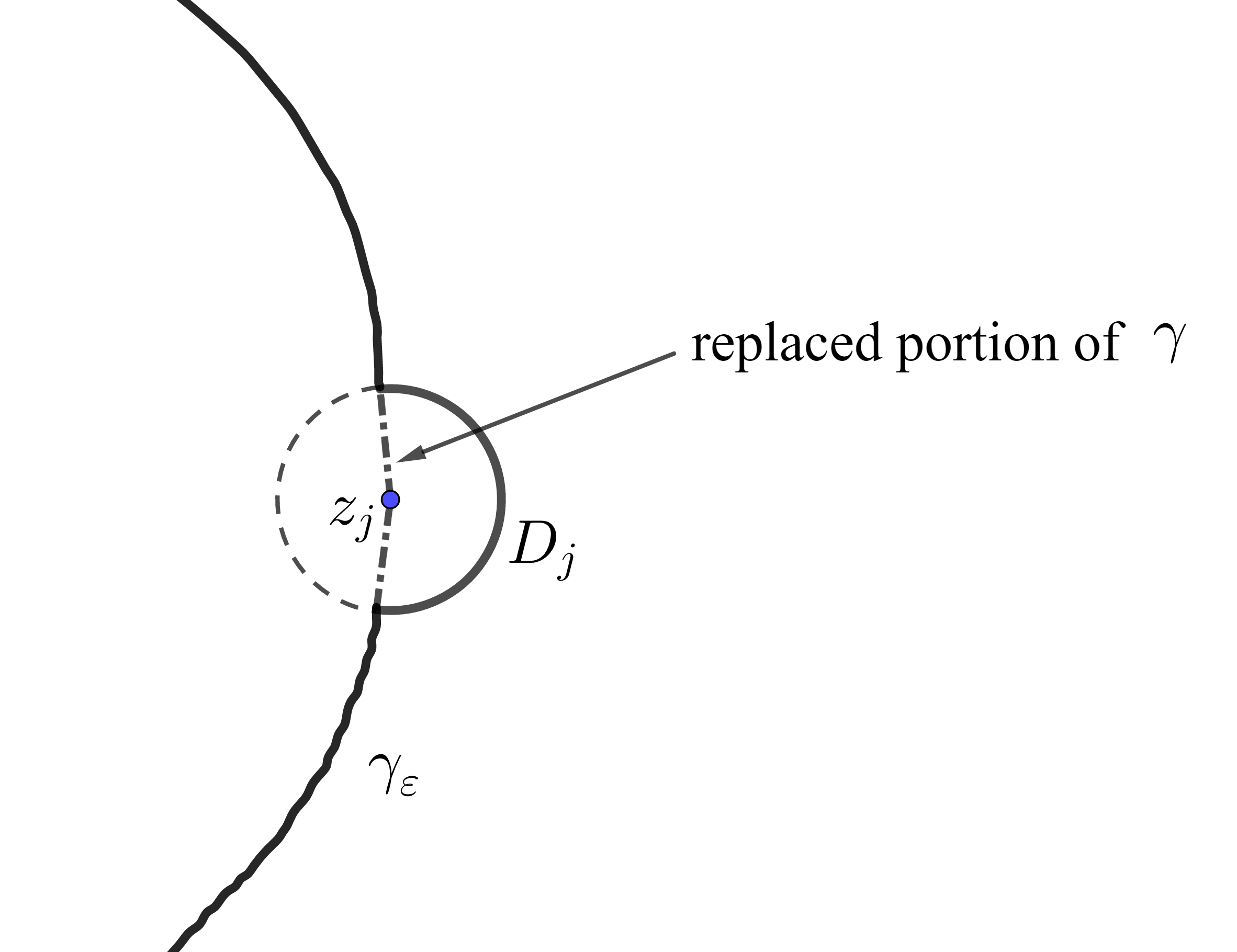}\caption{The bold curve is $\gamma_{\varepsilon}$. }
\end{figure}

No zeroes of $f$ lie on $\gamma_\varepsilon$, and there are $m+\lambda$ zeroes of $f$ inside $\gamma_\varepsilon$, so by Theorem \ref{classarg} there are at least $2m+2\lambda$ points on $\gamma_\varepsilon$ which are preimages of points in $L$. The trick is to show that not too many of them can be on the components of $\gamma_\varepsilon$ which lie in the boundary of some $\partial D_{j}$. The intuition behind this is easy: if we shrink $\varepsilon$ sufficiently, then the points on $\partial D_{j}$ which are preimages of points on $L$ should be approximately equidistributed on $\partial D_{j}$, so that about half of them lie inside and half lie outside $\gamma$. Only the ones on the outside of $\gamma$ lie on $\gamma_\varepsilon$, and any other points on $\gamma_\varepsilon$ which are preimages of points in $L$ must lie on $\gamma$ proper. The trick is making this rigorous.

Inside $\overline{D}_{j}$, $f$ is of the form $f(z)=(z-z_{j})^{\lambda_{j}}g_{j}(z)$
where $g_{j}(z_{j})\neq0$. If we let $z=z_{j}+\varepsilon e^{\theta i}$
then we have 
\[
\arg(f(z_{j}+\varepsilon e^{\theta i}))=\lambda_{j}\theta+\arg(g_{j}(z)).
\]

We claim that 
\begin{equation}
\mid{\textstyle \frac{d\arg(g_{j})}{d\theta}}(z_{j}+\varepsilon e^{\theta i})\mid\overset{\varepsilon}{=}o(1).\label{eq:o(eps)}
\end{equation}
Note that this implies from above that 

\begin{equation}
\mid{\textstyle \frac{d\arg(f)}{d\theta}}(z_{j}+\varepsilon e^{\theta i})\mid\overset{\varepsilon}{=} \lambda_j  + o(1).\label{eq:fo(eps)}
\end{equation}

We shall use the approximation (\ref{eq:fo(eps)}), and will prove (\ref{eq:o(eps)})
later on in a separate lemma. Now, we may shrink $\varepsilon$
if necessary so that $\frac{d\arg(f)}{d\theta}(z_{j}+\varepsilon e^{\theta i})$ is positive on all $\partial D_{j}$.
This fact, combined with Theorem \ref{classarg}, yields
that $\partial D_{j}$ contains exactly $2\lambda_{j}$ preimages of
points in $L$, since the argument of the curve may not change direction in order to create extra preimages. In other
words,
\[
\#\{z\in\partial D_{j}\mid f(z)\in L\}=2\lambda_{j}.
\]
Denote by $u_{0}=z_{j}+re^{\theta_{0}i},...,u_{\ell}=z_{j}+re^{\theta_{\ell}i}$
the points of $\gamma_{\varepsilon}\cap\partial D_{j}$ such that
$f(u_{t})\in L$, arranged in anti-clockwise order. Note that $\arg(f(u_{t}))-\arg(f(u_{t-1})) = \pi$ for all $t$, since $f$ is continuous and orientation preserving.  Using
the mean value theorem and (\ref{eq:fo(eps)}) we have 
\[
\begin{vmatrix}\frac{\arg(f(u_{t}))-\arg(f(u_{t-1}))}{\theta_{t}-\theta_{t-1}}\end{vmatrix}={\textstyle \frac{\pi}{\theta_{t}-\theta_{t-1}}}\leq\lambda_{j}+o(1),
\]
whence 
\[
\ell\leq \sum_{t=1}^\ell {\textstyle \frac{(\theta_{t}-\theta_{t-1})(\lambda_{j}+o(1 ))}{\pi}}=\textstyle \frac{(\lambda_{j}+o(1 ))}{\pi} (\theta_{\ell}-\theta_{0})\leq{\textstyle \frac{\lambda_{j}+\varepsilon\delta_{j}}{\pi}}(\pi+\eta_{\varepsilon}),
\]
where $\pi + \eta_{\varepsilon}$ denotes the angular length of the circular arc $\partial D_{j} \cap\gamma_\varepsilon$. Note that the smoothness of $\gamma$ implies that $\eta_{\varepsilon}\rightarrow0$ as $\varepsilon \to 0$.

\begin{figure}[H]
\begin{centering}
\includegraphics[width=6cm,height=6cm,keepaspectratio]{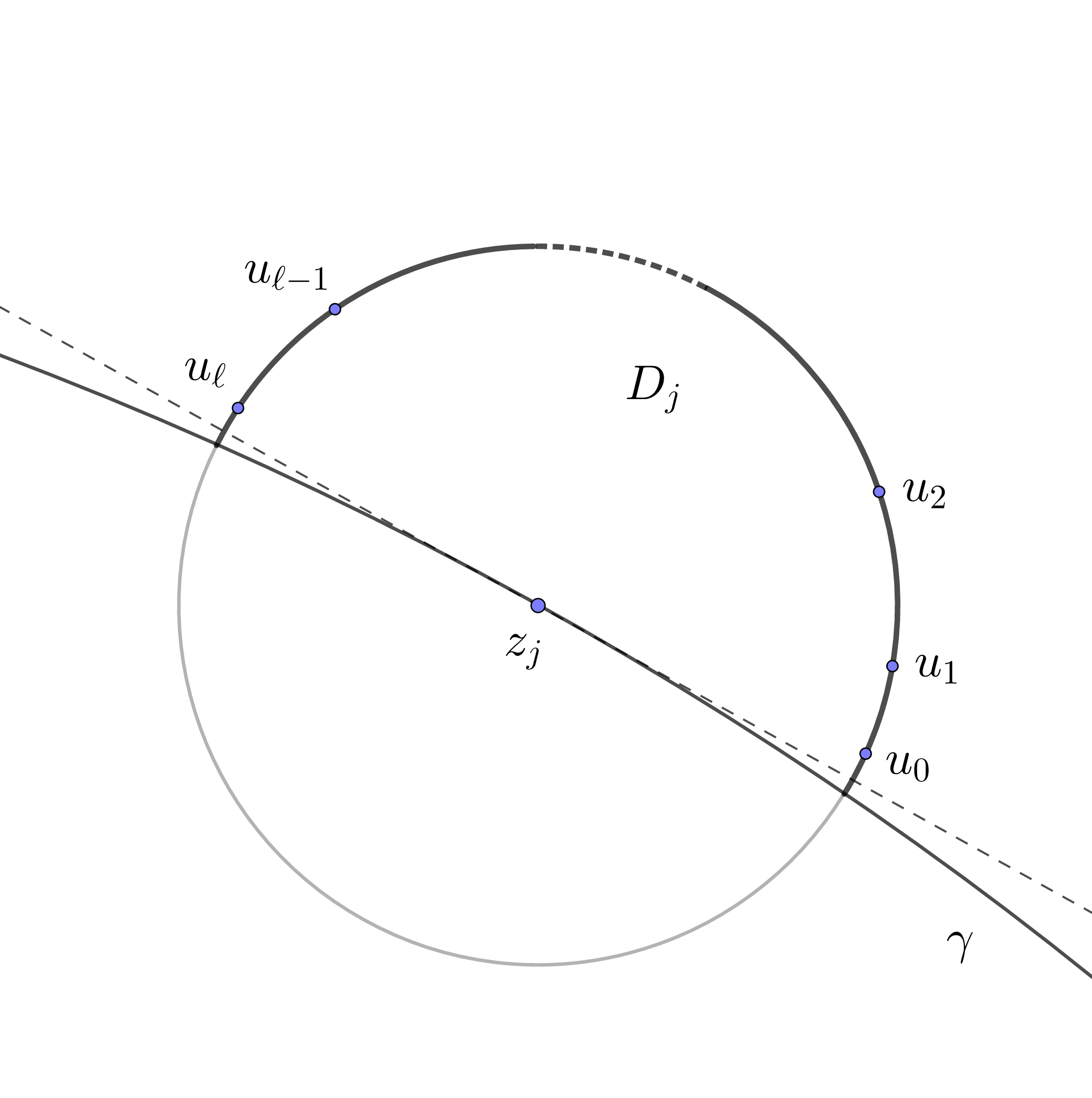}\caption{The dotted line is the tangent of $\gamma$ at $z_{j}$.}
\par\end{centering}
\end{figure}

Thus, by again shrinking $\varepsilon$ if necessary, we obtain 
\[
\ell\leq\lambda_{j}.
\]

There are $\ell+1$ points on $\partial D_{j} \cap\gamma_\varepsilon$ which are preimages of points in $L$, and applying this around every $z_j$ and combining these estimates yields at least

$$
2m+2\lambda - \sum_{j=1}^k (\lambda_j+1)
$$

points on $\gamma \cap \gamma_\varepsilon$ which are preimages of points in $L$. Recalling that $\lambda = \sum_{j=1}^k \lambda_j$ gives at least 

$$
2m+ \sum_{j=1}^k (\lambda_j-1)
$$

points on $\gamma \cap \gamma_\varepsilon$ which are preimages of points in $L$. Finally, the points $z_1, \ldots, z_k$ are all mapped to $L$ by $f$, since $L$ passes through the origin, and taking these into account completes the proof of the theorem.

\iffalse
there must Consequently, inside the curve $\gamma$, at least $\lambda_{j}-1$
points of $\partial D_{j}$ and are mapped to $L$. Summarizing all
contributions made by the $z_{j}$'s, the total number of points on
and inside $\gamma$ is bounded below by 
\[
\underset{\overset{\uparrow}{\mathrm{CAP}\,\mathrm{inside}\,\gamma}}{2m}+\sum_{j=1}^{k}(\lambda_{j}-1+\underset{\overset{\uparrow}{z_{j}}}{1})=2m+\sum_{j=1}^{k}\lambda_{j}=2m+\lambda
\]
\fi
\end{proof}

It remains only to prove the lemma used earlier on the derivative of the argument of $g$.

\begin{lem}
We have 
\[
\mid{\textstyle \frac{d\arg(g_{j})}{d\theta}}(z_{j}+\varepsilon e^{\theta i})\mid\overset{\varepsilon}{=}o(1).
\]
\end{lem}

\begin{proof}
The Taylor expansion of $g_{j}(z=z_{j}+\varepsilon e^{\theta i})$
at $z_{j}$ is given by 
\[
g_{j}(z_{j}+\varepsilon e^{\theta i})=g_{j}(z_{j})+\sum_{r=1}^{\infty}{\textstyle \frac{g_{j}^{(r)}(z_{j})}{r!}}\varepsilon^{r}e^{r\theta i}=g_{j}(z_{j})+\varepsilon\varphi_{j}(\theta).
\]
Notice that the function $\varphi_{j}$ is differentiable (with respect
to $\theta$) with bounded derivative. Set $g_{j}(z_{j})=a+bi$ and
$\varphi_{j}(\theta)=\varepsilon\alpha(\theta)+\varepsilon\beta(\theta)i$.
Without loss of generality we may assume that $a,b>0$ and hence for
small $\varepsilon$ we get $a+\varepsilon\alpha(\theta),b+\varepsilon\beta(\theta)>0$.
In particular, we can express $g_{j}(z_{j}+\varepsilon e^{\theta i})$
as 
\[
g_{j}(z_{j}+\varepsilon e^{\theta i})=\arctan({\textstyle \frac{b+\beta(\theta)}{a+\alpha(\theta)}}).
\]
Therefore
\[
\begin{alignedat}{1}{\textstyle \frac{d\arg(g_{j})}{d\theta}}(z) & ={\textstyle \frac{\varepsilon\beta'(\theta)(a+\varepsilon\alpha(\theta))-\varepsilon\alpha'(\theta)(b+\varepsilon\beta(\theta))}{(a+\varepsilon\alpha(\theta))^{2}}\times\frac{1}{1+({\textstyle \frac{b+\varepsilon\beta(\theta)}{a+\varepsilon\alpha(\theta)}})^{2}}}\\
 & =\mathcal{O}(\varepsilon)\\
 & =o(1).
\end{alignedat}
\]
\end{proof}

\section{Concluding remarks}

The condition that $\gamma$ be smooth can be weakened to piecewise smooth if required. In that case we can associate an interior angle $\beta_j$ and an interior angle $\alpha_j$ to each of the zeroes $z_j$; see the figure below. The same reasoning applies, except that an upper bound on the number of preimages of points on $L$ on each $\partial D_{j}$ is given by $\lfloor  \frac{\beta}{\pi}\lambda_j \rfloor +1$ where $\lfloor \cdot \rfloor$ is the floor function. A bit of algebra yields the following generalization of Theorem \ref{thm:Let--be}.

\begin{thm}
\label{thm:Let--be2} Let $f$ be a non zero holomorphic function defined
on an open domain $U$ and $\gamma$ be a piecewise smooth Jordan curve lying
inside $U$. Suppose the zeroes of $f$ lying on $\gamma$ are $z_1, \dots, z_k$. Let $\lambda_j$ be the multiplicity of the zero at $z_j$, and let $\alpha_j$ be the interior angle of $\gamma$ at $z_j$. Then for any line $L$ passing through the origin there exist at least 

$$
2m+\sum_{j=1}^k \lceil \lambda_j \textstyle \frac{\alpha_j}{\pi} \rceil
$$

distinct points on $\gamma$ mapped to $L$ by $f$, where
$m$ is the number of zeros of $f$ inside $\gamma$ counted according to
multiplicities and $\lceil \cdot \rceil$ denotes the ceiling function.
\end{thm}
\begin{figure}[H] \label{alpha}
\begin{centering}
\includegraphics[width=6cm,height=6cm,keepaspectratio]{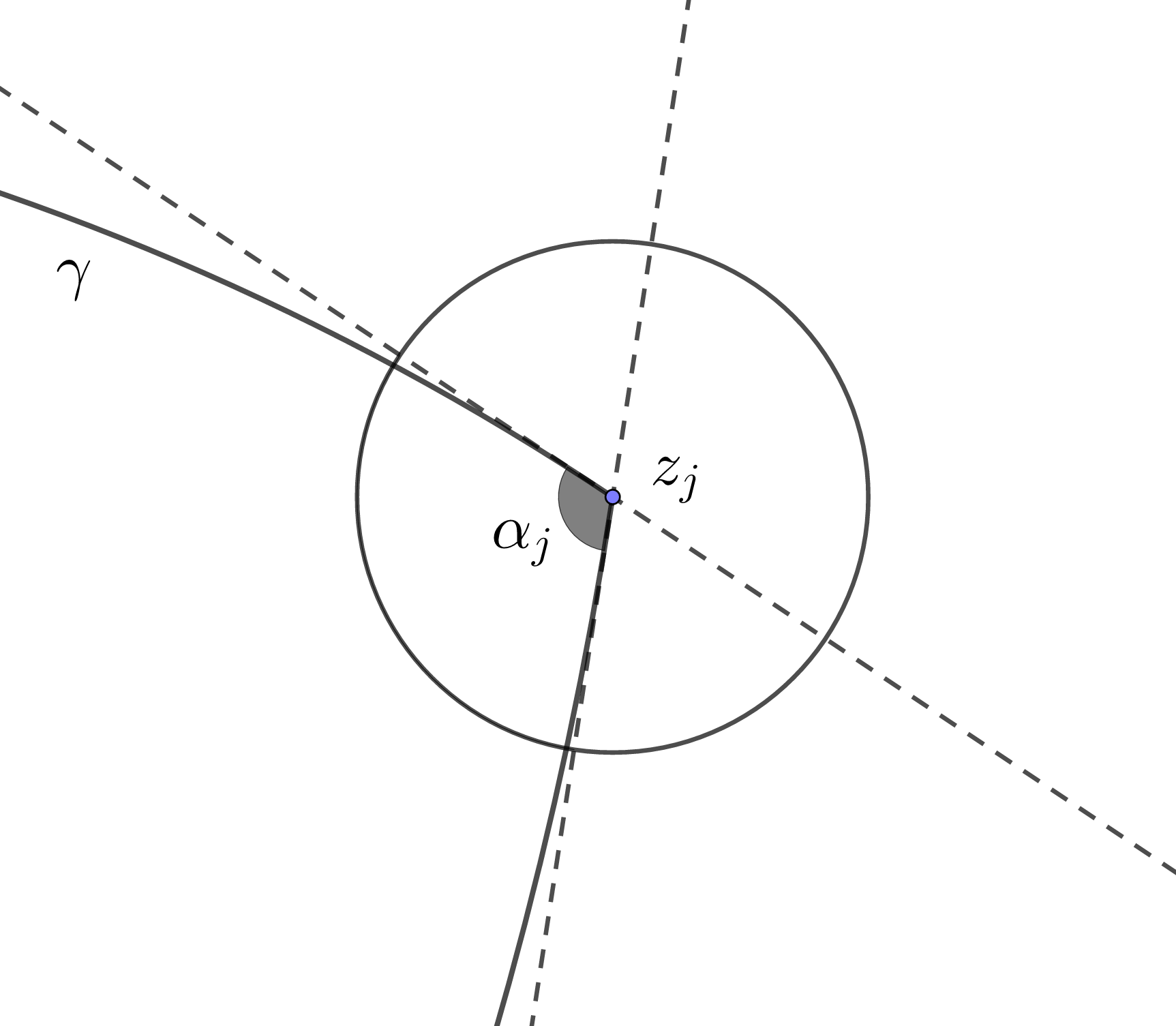}\caption{The angle $\alpha_j$ at $z_{j}$.}
\par\end{centering}
\end{figure}

\iffalse
The smoothness of $\gamma$ can be restricted only at around the points
$z_{1},...,z_{k}$. Using the same proof of our theorem, we can show
that an upper bound of the number of crossing is $2m+\lambda+2k$.
\fi

\bibliographystyle{plain}
\bibliography{Mref}
\end{document}